\documentclass[11pt]{amsart}
\usepackage{geometry}               
\usepackage{graphicx}
\usepackage{amssymb}
\usepackage{pdfsync}
\usepackage{hyperref}
\usepackage{epstopdf}
\def\R{\mathbb{R}}

\def\N{\mathbb{N}}

\newtheorem{theorem}{Theorem}

\newtheorem{definition}[theorem]{Definition}

\newtheorem{claim}[theorem]{Claim}
\newtheorem{question}[theorem]{Question}

\title{Nonexistence of extremizers for certain convex curves}
\author{Diogo Oliveira e Silva}
\address{
        Diogo Oliveira e Silva\\
        Department of Mathematics\\
        University of California \\
        Berkeley, CA 94720-3840, USA}
\email{dosilva@math.berkeley.edu}
\thanks{The  author was partially supported by the Funda\c{c}\~{a}o para a Ci\^{e}ncia e a Tecnologia 
(FCT/Portugal grant SFRH/BD/28041/2006) 
and the National Science Foundation under agreement DMS-0901569.}

\date{\today}                                           

\begin{document}
\maketitle

\begin{abstract}
{\noindent We establish the nonexistence of extremizers for a local Fourier restriction inequality on a certain class of planar convex curves whose curvature satisfies a natural assumption. We accomplish this by studying the local behavior of the triple convolution of the arclength measure on the curve with itself, and show in particular that every extremizing sequence concentrates at a point on the curve.}
\end{abstract}

\section{Introduction}

This paper continues the study initiated in \cite{OS}, where we considered a certain class of smooth, convex, planar curves $\Gamma\subset\R^2$ equipped with arclength measure $\sigma$ and parametrized (not necessarily by arclength) by $\gamma$. If the curvature $\kappa$ of $\Gamma$ is positive everywhere, then the Tomas-Stein inequality \cite{T} states that there exists a finite constant $C<\infty$ such that
\begin{equation}\label{TS}
\|\widehat{f\sigma}\|_{L^6(\R^2)}\leq C \|f\|_{L^2(\sigma)}
\end{equation}
for every $f\in L^2(\sigma)$. With the Fourier transform defined as $\widehat{g}(\xi)=\int e^{-ix\cdot \xi} g(x)dx$, denote by

$${\mathbf C}[{\Gamma}]=\sup_{0\neq f\in L^2(\sigma)} \|\widehat{f\sigma}\|_6 \|f\|_{L^2(\sigma)}^{-1}$$ 
the optimal constant in the inequality \eqref{TS}. By an extremizing sequence for the inequality \eqref{TS} we mean a sequence $\{f_n\}$ of functions in $L^2(\sigma)$ satisfying $\|f_n\|_2\leq 1$ and such that $\|\widehat{f_n \sigma}\|_6\rightarrow {\mathbf{C}[\Gamma]}$ as $n\rightarrow\infty$. An extremizer for the inequality \eqref{TS} is then a function $f\neq 0$ in $L^2(\sigma)$ which satisfies $\|\widehat{f\sigma}\|_6=\mathbf{C}[\Gamma]\|f\|_2$.

Assume that $\Gamma$ has no points with collinear tangents, and that its curvature $\kappa$ satisfies\footnote{Derivatives in \eqref{3/2geocond} are  taken with respect to the natural arclength parameter. }
\begin{equation}\label{3/2geocond}
(\kappa''\cdot\kappa^{-3})(p)<\frac{3}{2},
\end{equation}
for every point $p\in \Gamma$ at which  $\kappa$ attains a global minimum. 
In this case, \cite{OS} establishes in particular the existence of extremizers for the corresponding Fourier restriction inequality \eqref{TS}. More generally, extremizing sequences of nonnegative functions are shown to always have a subsequence which converges to an extremizer.

It is natural to ask about the significance of the geometric condition \eqref{3/2geocond}. For instance, if it is not satisfied, does this mean that extremizers  fail to exist?
While we are unable to provide a complete answer to this question, in this paper we consider the case in which condition \eqref{3/2geocond} fails in a rather strong sense,
and prove a complementary (negative) result along these lines. Before we state it precisely, we need a definition:

\begin{definition}
A sequence of functions $\{f_n\}$ in $L^2(\sigma)$ satisfying $\|f_n\|_2\rightarrow 1$ as $n\rightarrow\infty$ concentrates at a point $p=\gamma(z)\in\Gamma$ if for every $\epsilon, r>0$ there exists $N\in\N$ such that, for every $n\geq N$,
$$\int_{|w-z|\geq r} |f_n(\gamma(w))|^2 \|\gamma'(w)\| dw <\epsilon.$$
\end{definition}

\noindent As pointed out in \cite{L, CS,  FVV1, Q1, Q2, OS}, the possibility that an extremizing sequence might concentrate is a major obstruction to its precompactness and therefore to the existence of extremizers. The most involved step in \cite{OS} consisted in ruling out concentration of an extremizing sequence of nonnegative functions at a point on the curve.

 Here is our main result:

\begin{theorem}\label{extremizersdonotexist}
Let $\Gamma$ be a compact arc of a smooth, convex curve in the plane, equipped with arclength measure $\sigma$. Assume that the curvature $\kappa$ of $\Gamma$ is a strictly positive function. If the second derivative of the curvature with respect to arclength satisfies
\begin{equation}\label{3geocond}
\frac{d^2\kappa}{ds^2}(p_0)>3\kappa(p_0)^3
\end{equation}
at some $p_0\in\Gamma$ which is a global minimum of the curvature, then 
there exist an extremizing sequence which concentrates at $p_0$ and a neighborhood of $p_0$ for which the local analogue of inequality \eqref{TS} has no extremizers.  
\end{theorem}

 {\bf Notation.}
If $x,y$ are real numbers, we will write $x=O(y)$ or $x\lesssim y$ if there exists a finite constant $C$ such that $|x|\leq C|y|$. If we want to make explicit the dependence of the constant $C$  on some parameter $\alpha$, we will  write $x=O_\alpha(y)$ or $x\lesssim_\alpha y$. As is customary the constant $C$ is allowed to change from line to line.
If $\lambda\in\R$ and $A\subseteq\R^d$, we denote its $\lambda$-dilation by $\lambda\cdot A:=\{\lambda x: x\in A\}$. 
Sharp constants will always appear in bold face.
\vspace{.5cm}

{\bf Acknowledgments.} The author would like to thank his dissertation advisor, Michael Christ, for many helpful discussions. 

\section{Proof of Theorem \ref{extremizersdonotexist}}

As in \cite[\S 8.1]{OS}, we start by introducing suitable local coordinates. Let $p_0\in\Gamma$ be a point of minimum curvature. After a suitable translation we lose no generality in assuming that $p_0=(0,0)$. Possibly after rotation, a neighborhood of $p_0$ in the arc $\Gamma$ can then be parametrized in the following way:\footnote{
There is no cubic term in the expression for $g$ because by assumption the curvature has a minimum at $\gamma(0)=(0,0)=p_0$. Constant and linear terms were likewise removed via the affine change of variables just described.}
\begin{displaymath}
\begin{array}{rcl}\label{parametrization}
\gamma: [-r,r] & \rightarrow & \mathbb{R}^2\\
y & \mapsto &  \gamma(y)=\Big(y,g(y)=\frac{\lambda}{2}y^2+ay^4+\varphi(y)\Big),
\end{array}
\end{displaymath}
where $r>0$ and $\lambda:=\kappa(p_0)=\min_\Gamma \kappa>0$; the function $\varphi$ is real-valued, smooth, and it satisfies $\varphi(y)=O(|y|^5)$ as $|y|\rightarrow 0$; the parameter $a$ is a function of the second derivative of the curvature with respect to arclength at 0. Indeed, it can be readily checked that
$$\frac{d^2\kappa}{ds^2}(0)=24a-3\lambda^3,$$
and so hypothesis \eqref{3geocond} in Theorem \ref{extremizersdonotexist} is equivalent to $a>2(\frac{\lambda}{2})^3$.
Incidentally\footnote{We shall abuse notation slightly and henceforth denote by $(\Gamma,\sigma)$ the $r$-neighborhood of $p_0=(0,0)$ parametrized by $\gamma$ as indicated, equipped with the corresponding arclength measure.}
 note that the curvature $\kappa$ of $\Gamma$ at a point $(y,g(y))\in\Gamma$ is given by
$$\kappa(y)=\frac{g''(y)}{(1+g'(y)^2)^{3/2}}=\frac{\lambda+12 a y^2+\varphi''(y)}{(1+\lambda^2 y^2+8a \lambda y^4+16 a^2 y^6+2\lambda y\varphi'(y)+8a y^3\varphi'(y)+\varphi'(y)^2)^{3/2}},$$
and one can easily check that $\kappa$ attains a minimum at the origin $(0,0)\in\Gamma$ if and only if $a\geq(\frac{\lambda}{2})^3$.

The Tomas-Stein inequality states that there exists a finite constant $\mathbf C=\mathbf C_{r,\lambda,a,\varphi}<\infty$ such that
\begin{equation}\label{specialcaprestriction}
\|\widehat{f\sigma}\|_{L^6(\R^2)}\leq{\mathbf C}\|f\|_{L^2(\sigma)},\;\;\textrm{ for every }f\in L^2(\sigma);
\end{equation}
\noindent  by ${\mathbf C}=\sup_{0\neq f\in L^2}\|\widehat{f\sigma}\|_6 \|f\|_2^{-1}$ we mean the optimal constant as usual. Extremizing sequences and extremizers for inequality \eqref{specialcaprestriction} can be defined in a completely analogous way as before.

Theorem \ref{extremizersdonotexist} will follow once we establish the following result:

\begin{theorem}\label{extremizersdonotexistmodified}
Let $\lambda, a>0$ be such that $a>2(\frac{\lambda}{2})^3$, and consider the curve $({\Gamma},{\sigma})$ parametrized by $\gamma$ as above (in terms of $r,\lambda, a$ and $\varphi$). Then there exists $r_0>0$ such that for every $r<r_0$, the triple convolution ${\sigma}\ast{\sigma}\ast{\sigma}$ attains a strict global maximum at the origin.
As a consequence, if $r<r_0$, every extremizing sequence concentrates at the origin. In particular, there are no extremizers for inequality \eqref{specialcaprestriction}.  
\end{theorem}

To prove Theorem \ref{extremizersdonotexistmodified}, the strategy will be to analyze the explicit formula for the triple convolution $\sigma\ast\sigma\ast\sigma$ obtained in \cite[\S 7]{OS} and use it to show that a suitable modification thereof is sufficiently regular inside its support.\footnote{For an analogous reasoning in the bilinear setting, see \cite{BMcM} and the references therein, especially \cite{Fef}.} We compute its partial derivatives of order 2 at the critical point $(0,0)$ in order to apply the second derivative test.  The last  assertions of the theorem follow via an argument based on ideas of Foschi \cite{F} and Quilodr\'an \cite{Q2}.

\begin{proof}[Proof of Theorem \ref{extremizersdonotexistmodified}] 
Let $\eta_r\in C_0^\infty(\R)$ be a mollified version of the characteristic function of the interval $[-r,r]$: to accomplish this, fix $\eta\in C_0^\infty(\R)$ such that $\eta\equiv 1$ on $[-1,1]$ and $\eta(\xi)=0$ if $|\xi|\geq 2$, and define $\eta_r:=\eta(r^{-1}\cdot)$.
If $|\xi|\leq r$, we have that 
$$\sigma(\xi,\tau)=G_r(\xi)\delta(\tau-g(\xi))d\xi d\tau,$$
where $G_r(y):=(1+g'(y)^2)^{1/2}\eta_r(y)$ is a smooth function inside its support $\{y\in\R: |y|\lesssim r\}$. 

For $\theta\in [0,2\pi]$, consider the quantities
$$\alpha(\theta)=-\frac{\cos\theta}{\sqrt{2}}-\frac{\sin\theta}{\sqrt{6}},\; \beta(\theta)=\frac{\cos\theta}{\sqrt{2}}-\frac{\sin\theta}{\sqrt{6}}\textrm{ and }\; \gamma(\theta)=2\frac{\sin\theta}{\sqrt{6}},$$
which satisfy the relations
\begin{align}
&\alpha(\theta)+\beta(\theta)+\gamma(\theta)=0;\label{abc1}\\
&\alpha(\theta)^2+\beta(\theta)^2+\gamma(\theta)^2=1;\label{abc2}\\
&\alpha(\theta)^3+\beta(\theta)^3+\gamma(\theta)^3=-{\sin(3\theta)}/\sqrt{6};\label{abc3}\\
&\alpha(\theta)^4+\beta(\theta)^4+\gamma(\theta)^4={1}/{2}.\label{abc4}
\end{align}

 Let $|\xi|\leq r$ and $0\leq\rho\lesssim r$. The following functions will appear in the expression \eqref{tripleconvolution} for $\sigma\ast\sigma\ast\sigma$ below:
\begin{equation}\label{quarticG}
\widetilde{\mathfrak{G}}_r(\xi;\rho,\theta):=G_r\Big(\frac{\xi}{3}+\alpha(\theta)\rho\Big)G_r\Big(\frac{\xi}{3}+\beta(\theta)\rho\Big)G_r\Big(\frac{\xi}{3}+\gamma(\theta)\rho\Big);
\end{equation}

\begin{multline}
\psi(\xi;\rho,\theta):=\frac{\lambda}{2}\rho^2+\frac{2}{3}a\xi^2\rho^2-\Big(\frac{2}{3}\Big)^{3/2}a\xi\sin(3\theta)\rho^3+\frac{a}{2}\rho^4\label{quarticpsi}\\
+\varphi\Big(\frac{\xi}{3}+\alpha(\theta)\rho\Big)+\varphi\Big(\frac{\xi}{3}+\beta(\theta)\rho\Big)+\varphi\Big(\frac{\xi}{3}+\gamma(\theta)\rho\Big)-3\varphi(\xi/3).
\end{multline}

\noindent Using Taylor's formula and relations \eqref{abc1}$-$\eqref{abc4}, we can rewrite the summands of the function $\psi$ which involve the higher order term $\varphi$ as:
\begin{multline}
\varphi\Big(\frac{\xi}{3}+\alpha(\theta)\rho\Big)+\varphi\Big(\frac{\xi}{3}+\beta(\theta)\rho\Big)+\varphi\Big(\frac{\xi}{3}+\gamma(\theta)\rho\Big)-3\varphi(\xi/3)=\frac{\varphi''(\xi/3)}{2}\rho^2\label{phiTaylor}\\
-\frac{\varphi^{(3)}(\xi/3)}{6\sqrt{6}}\sin(3\theta)\rho^3+\frac{\varphi^{^{(4)}}(\xi/3)}{48}\rho^4+O_{\xi,\theta}(\rho^5).
\end{multline}

The function $\widetilde{\mathfrak{G}}_r$ is smooth inside its support, which is compact in all three variables. It will not give us much trouble.
On the other hand, the function $\psi$  defines a smooth function in all its variables, but we will need to somehow invert it. This is a somewhat delicate procedure, a variant of which already showed up in \cite[Appendix 1]{OS}. 

 From the proof of \cite[Proposition 24]{OS}, we have that the triple convolution $\sigma^{(\ast 3)}:=\sigma\ast\sigma\ast\sigma$ is given by
\begin{equation}\label{tripleconvolution}
\sigma^{(\ast 3)}  (\xi,\tau)=\frac{\chi(|\xi|\lesssim r)}{\sqrt{3}}\int_0^{2\pi}\int_{0\leq\rho\lesssim r}\widetilde{\mathfrak{G}}_r(\xi;\rho,\theta)\delta\Big(\tau-3 g({\xi}/{3})-\psi(\xi;\rho,\theta)\Big)\rho d\rho d\theta.
\end{equation}

\noindent We introduce a change of variables which is related but not identical to the one described in the proof of  \cite[Proposition 24]{OS}.
Start by noting that, for $v\geq 0$, the equation $\psi(\xi;\rho,\theta)=v^2$ can be uniquely solved for $\rho=\rho_{\xi,\theta}(v)\geq 0$ in a sufficiently small right half-neighborhood of the origin. Recalling \eqref{phiTaylor}, we may write the equation in the equivalent form 
\begin{multline*}
v=v_{\xi,\theta}(\rho)=\\
=\rho\Big(\frac{\lambda}{2}+\frac{2}{3}a\xi^2+\frac{\varphi''(\xi/3)}{2}-\Big((\frac{2}{3})^{3/2}a\xi+\frac{\varphi^{(3)}(\xi/3)}{6\sqrt{6}}\Big)\sin(3\theta)\rho+\Big(\frac{a}{2}+\frac{\varphi^{(4)}(\xi/3)}{48}\Big)\rho^2+O(\rho^3)\Big)^{1/2}.
\end{multline*}
One sees that $v$ defines a $C^\infty$ diffeomorphism from $[0,r]$ onto its image, provided $r>0$ is chosen sufficiently small (as a function of $\lambda, a$ and $\varphi$).
Moreover, $r$ can be chosen to ensure  
$$\frac{\partial v}{\partial\rho}(\rho)>0\textrm{ if }\rho\in [0,r].$$

\noindent By the inverse function theorem,
$\rho=\rho_{\xi,\theta}(v)$ will then define a smooth function of $v$ on $v([0,r])\subset [0,\infty)$, and $\partial_v \rho(v)>0$ on the same interval. In particular,
\begin{equation}\label{rhoprimeatzero}
\rho'_{\xi,\theta}(0)=\frac{1}{v'_{\xi,\theta}(0)}=\frac{1}{(\lambda/2+2/3a\xi^2+\varphi''(\xi/3)/2)^{1/2}}>0.
\end{equation}

\noindent Changing variables $\rho=\rho_{\xi,\theta}(v)$ in \eqref{tripleconvolution}, we have that
\begin{multline}
\sigma^{(\ast 3)}  (\xi,\tau)=\frac{\chi(|\xi|\lesssim r)}{\sqrt{3}}\int_0^{2\pi}\int_{0\leq v\leq C\psi^{1/2}}\widetilde{\mathfrak{G}}_r(\xi;\rho_{\xi,\theta}(v),\theta)\cdot\label{firstpart}\\
\cdot\delta\Big(\tau-3g({\xi}/{3})-v^2\Big)
\frac{\rho_{\xi,\theta}(v)}{\partial_\rho \psi(\xi;\rho_{\xi,\theta}(v), \theta)} 2v dv d\theta,
\end{multline}
and so

\begin{equation}\label{triplesqrt}
\sigma^{(\ast 3)} (\xi,\tau)=\frac{\chi(|\xi|\lesssim r)}{\sqrt{3}}\int_0^{2\pi}  \widetilde{\mathfrak{G}}_r(\xi;\rho_{\xi,\theta}(\sqrt{\tau-3g({\xi}/{3})}),\theta) \frac{\rho_{\xi,\theta}(\sqrt{\tau-3g({\xi}/{3})})}{\partial_\rho\psi(\xi;\rho_{\xi,\theta}(\sqrt{\tau-3g({\xi}/{3})}),\theta)}d\theta.
\end{equation}
For $\theta\in [0,2\pi]$, the integrand in \eqref{triplesqrt} is supported in the region
$$\{(\xi,\tau)\in\R^2: 0\leq {\tau-3g({\xi}/{3})}\leq C^2\psi(\xi;r,\theta)\},$$
where the constant $C<\infty$ is large enough that the restriction $v\leq C\psi(\xi;r,\theta)^{1/2}$ in the inner integral of \eqref{firstpart} becomes redundant because of support limitations on factors present in the integrand. 

As observed in the course of the proof of \cite[Proposition 24]{OS}, expression \eqref{triplesqrt} defines a continuous function of the variables $\xi,\tau$. However, it is {\em not}  a differentiable function of $\tau$ at $\tau=0$.
To remedy this, introduce a new parameter $\epsilon\geq 0$ defined by the equation
$$\epsilon^2:=\tau-3g({\xi}/{3})=\tau-\frac{\lambda}{6}\xi^2-\frac{a}{27}\xi^4-3\varphi({\xi}/{3}).$$ 
The geometric significance of $\epsilon$ is clear: for $|\xi|\leq 3r$ and $\tau\geq 3g(\xi/3)$, it measures the (square root of the) vertical distance from the point $(\xi,\tau)\in\R^2$ to the ``lower boundary'' of the support of the triple convolution $\sigma^{(\ast 3)}$, given by
$$\Big\{(\xi,\tau)\in\R^2: \tau=3g({\xi}/{3})\Big\}.$$ 
For $\epsilon\geq 0$, define the function

\begin{equation}\label{tripleconvolutiontilde}
F (\xi,\epsilon):=\frac{\chi(|\xi|\lesssim r)}{\sqrt{3}}\int_0^{2\pi}  \widetilde{\mathfrak{G}}_r(\xi;\rho_{\xi,\theta}(\epsilon),\theta) \frac{\rho_{\xi,\theta}(\epsilon)}{\partial_\rho\psi(\xi;\rho_{\xi,\theta}(\epsilon),\theta)}d\theta.
\end{equation}

\noindent In view of \eqref{triplesqrt}, a sufficient condition for the convolution $\sigma\ast\sigma\ast\sigma$ to have a strict local maximum at $(\xi,\tau)=(0,0)$ is that the function $F$ has a strict local maximum at $(\xi,\epsilon)=(0,0)$.
The advantage of considering  \eqref{tripleconvolutiontilde} instead of \eqref{triplesqrt} lies in its extra regularity, which is needed to justify the computations that will follow:

\begin{claim}\label{C^2}
There exists  $r_1>0$ such that expression \eqref{tripleconvolutiontilde} defines a $C^\infty$ function of both $\xi$ and $\epsilon$ in the rectangle
$$
\{(\xi,\epsilon)\in\R^2: |\xi|\leq r_1, 0\leq \epsilon\leq r_1\}.$$
\end{claim}

\begin{proof}[Proof of Claim \ref{C^2}] 
All the work has basically been done. 
Since $\rho$ is a smooth function of $\epsilon$,
the function $\widetilde{\mathfrak{G}}_r$ is smooth in the variables $\xi,\epsilon$ and $\theta$; it is also  compactly supported in all its variables. 
Noting that $\rho(\epsilon)=O(\epsilon)$, we compute:
\begin{multline}
\frac{\rho_{\xi,\theta}(\epsilon)}{\partial_{\rho}\psi(\xi;\rho_{\xi,\theta}(\epsilon),\theta)}=\label{rho/drho}\\
=\frac{1}{\lambda+\frac{4}{3} a \xi^2+\varphi''(\xi/3)-\Big(\frac{2\sqrt{2}}{\sqrt{3}}a\xi+\frac{\varphi^{(3)}(\xi/3)}{2\sqrt{6}}\Big)\sin (3\theta) \rho_{\xi,\theta}(\epsilon) +\Big(2 a+\frac{\varphi^{(4)}(\xi/3)}{12}\Big) \rho_{\xi,\theta}(\epsilon)^2+O(\epsilon^3)}.
\end{multline}
If $r_1>0$ is chosen sufficiently small, then $\rho/\partial_\rho\psi(\rho)$  defines a positive, smooth function of $\xi$ (trivial) and $\epsilon$ which is bounded above uniformly by $1/\lambda$. 
Indeed, the denominator in \eqref{rho/drho} is never zero as long as $\xi,\epsilon$ are small enough, and moreover we have that

\begin{equation*}
\frac{4}{3}a \xi^2 +2 a \rho_{\xi,\theta}(\epsilon)^2\geq 2\frac{2\sqrt{2}}{\sqrt{3}}a|\xi| \rho_{\xi,\theta}(\epsilon)\geq \frac{2\sqrt{2}}{\sqrt{3}}a|\xi|\sin (3\theta) \rho_{\xi,\theta}(\epsilon).
\end{equation*}
This shows that, up to cubic terms, the denominator in \eqref{rho/drho} is  equal to $\lambda$ plus a nonnegative term, as claimed.
\end{proof}
 
We now compute the Hessian $D^2 F$ of the function $F$ at the critical point $(\xi,\epsilon)=(0,0)$. For this computation,  we lose no generality in assuming, as we will, that $\varphi\equiv 0$. In fact,  when we compute {first} and {second} derivatives of $F$ at the {origin}, all the terms involving the function $\varphi$ vanish.\footnote{Recall that $\varphi(y)=O(|y|^5)$ as $|y|\rightarrow 0$.} The details are straightforward to verify.

Consider  the  case $\xi=0$. Then $\psi(0;\rho,\theta)$ does not depend on $\theta$ and can be explicitly inverted.\footnote{This algebraic trick simplifies matters greatly and is not available if, say, the function $g$ contains a quintic term of the form $by^5$.} In fact, $\psi(0;\rho,\theta)=\psi(\rho)=\frac{\lambda}{2}\rho^2+\frac{a}{2}\rho^4$ and $\partial_\rho \psi(\rho)=\lambda \rho+ 2 a \rho^3$.
Since $\psi(\rho)=\epsilon^2$ and $\rho\geq 0$,
\begin{equation}\label{expressionforrho}
\rho(\epsilon)=\sqrt{-\frac{\lambda}{2a}+\frac{\sqrt{\lambda^2+8 a \epsilon^2}}{2a}}.
\end{equation}
It follows that 
\begin{equation}\label{expressionforquotientsofrho}
\frac{\rho(\epsilon)}{\partial_\rho \psi(\rho(\epsilon))}=
\frac{1}{\lambda+2a\rho(\epsilon)^2}=\frac{1}{\sqrt{\lambda^2+8 a \epsilon^2}};
\end{equation}
as noticed before, this defines a smooth, positive function of $\epsilon$ which is bounded above uniformly by
$1/\lambda$. 
In particular, for $\epsilon\leq C\psi(r)^{1/2}$, we have from \eqref{tripleconvolutiontilde} and \eqref{expressionforquotientsofrho} that

 \begin{equation}\label{tripleatxi0}
F(0,\epsilon)=\frac{1}{\sqrt{3}}\int_0^{2\pi}  \widetilde{\mathfrak{G}}_r(0;\rho(\epsilon),\theta) \frac{\rho(\epsilon)}{\partial_\rho\psi(\rho(\epsilon))}d\theta= \frac{1}{\sqrt{3}\sqrt{\lambda^2+8 a \epsilon^2}}\int_0^{2\pi}  \widetilde{\mathfrak{G}}_r(0;\rho(\epsilon),\theta) d\theta.
\end{equation}

\noindent To the best of our knowledge, the integral in this last expression
cannot be explicitly evaluated. We can, however, Taylor expand it. For $|\zeta|\leq r$, we have that $\eta_r(\zeta)\equiv 1$ and 
\begin{equation}\label{TaylorGr}
G_r(\zeta)=G(\zeta)=({1+g'(\zeta)^2})^{1/2}=({1+(\lambda\zeta+4 a \zeta^3)^2})^{1/2}=1+\frac{\lambda^2}{2}\zeta^2+O(\zeta^4).
\end{equation}
By \eqref{quarticG}, we have that
$$\widetilde{\mathfrak{G}}_r(0;\rho(\epsilon),\theta)=G_r(\alpha(\theta)\rho(\epsilon))\cdot G_r(\beta(\theta)\rho(\epsilon))\cdot G_r(\gamma(\theta)\rho(\epsilon)).$$
Using the approximation given by \eqref{TaylorGr} on each of these three factors, recalling the relation \eqref{abc2}, and integrating with respect to the angular variable $\theta$, we get that
$$\int_0^{2\pi}  \widetilde{\mathfrak{G}}_r(0;\rho(\epsilon),\theta)d\theta=2\pi\Big(1+\frac{\lambda^2}{2}\rho(\epsilon)^2+O(\rho(\epsilon)^4)\Big)$$

\noindent if $\epsilon\leq C\psi(r)^{1/2}$ and $r>0$ is sufficiently small. Plugging this  into \eqref{tripleatxi0} and recalling that $\rho(\epsilon)=O(\epsilon)$, one gets that

 \begin{equation}\label{secondapproximation}
F (0,\epsilon)=\frac{2\pi}{\sqrt{3}}\cdot \frac{1}{\sqrt{\lambda^2+8 a \epsilon^2}}\Big(1+\frac{\lambda^2}{2}\rho(\epsilon)^2\Big)+O(\epsilon^4)\;\;\textrm{ if } \epsilon\leq C\psi(r)^{1/2}.
\end{equation}

\noindent We have an explicit expression for $\rho(\epsilon)$ given by \eqref{expressionforrho}. Using it, one finally obtains
\begin{equation}\label{firstderivative}
\frac{\partial}{\partial\epsilon}\Big|_{\epsilon=0^+}F(0,\epsilon)=0
\end{equation}
and
\begin{equation}\label{secondderivative}
\frac{\partial^2}{\partial\epsilon^2}\Big|_{\epsilon=0^+}F(0,\epsilon)=2\frac{2\pi}{\sqrt{3}}\Big(1-\frac{4a}{\lambda^3}\Big).
\end{equation}

\noindent Expression \eqref{secondderivative} defines a negative quantity if and only if $a>2(\frac{\lambda}{2})^3$. This turns out to be a necessary and sufficient condition for  the  function $F$ to have a strict local maximum at the origin $(0,0)$, provided $r>0$ is sufficiently small. 
Let us verify this in detail, thus establishing the first part of Theorem \ref{extremizersdonotexistmodified}.

Let $a>2(\frac{\lambda}{2})^3$.
We start by noting that \eqref{firstderivative} is valid also when $\xi\neq 0$ is sufficiently small. Using expression \eqref{tripleconvolutiontilde}, one can easily check that
\begin{equation}\label{depszero}
\frac{\partial}{\partial\epsilon}\Big|_{\epsilon=0^+}F(\xi,\epsilon)=0\;\;\textrm{ if } |\xi|\lesssim r.
\end{equation}
Indeed, the derivative in $\epsilon$ may fall in one of two factors. There is no contribution from the function $\widetilde{\mathfrak{G}}_r$ because equations \eqref{abc1} and \eqref{rhoprimeatzero} imply

$$\frac{\partial}{\partial\epsilon}\Big|_{\epsilon= 0^+} \widetilde{\mathfrak{G}}_r(\xi;\rho_{\xi,\theta}(\epsilon),\theta)=\rho_{\xi,\theta}'(0)(\alpha(\theta)+\beta(\theta)+\gamma(\theta))G_r'({\xi}/{3})G_r({\xi}/{3})^2=0.$$
The factor containing the quotient $\rho/\partial_\rho\psi(\rho)$ likewise does not contribute: we can use \eqref{rho/drho}, \eqref{rhoprimeatzero}, recall that $\rho_{\xi,\theta}(0)=0$, and conclude that

$$\frac{\partial}{\partial\epsilon}\Big|_{\epsilon= 0^+}\frac{\rho_{\xi,\theta}(\epsilon)}{\partial_\rho\psi(\rho_{\xi,\theta}(\epsilon))}=\frac{4}{\sqrt{3}}a\xi \Big(\lambda+\frac{4}{3}a\xi^2\Big)^{-5/2}\cdot \sin 3\theta.$$ 
The integral of the function $\sin 3\theta$ over the interval $[0,2\pi]$ vanishes, and this concludes the verification of \eqref{depszero}.
As a consequence, and since $F$ defines in particular a $C^2$ function of $\xi$ and $\epsilon$,
\begin{equation}\label{mixedpartials}
\frac{\partial^2}{\partial\xi\partial\epsilon}\Big|_{(\xi,\epsilon)=(0,0^+)}F(\xi,\epsilon)=0=\frac{\partial^2}{\partial\epsilon\partial\xi}\Big|_{(\xi,\epsilon)=(0,0^+)}F(\xi,\epsilon).
\end{equation}

\noindent At the lower boundary of the support of $\sigma^{(\ast 3)}$ (i.e. when $|\xi|\lesssim r$ and $\epsilon=0$), we have that

\begin{multline}
F(\xi,0)=\frac{1}{\sqrt{3}}\int_0^{2\pi}  \widetilde{\mathfrak{G}}_r(\xi;\rho_{\xi,\theta}(0),\theta) \frac{\rho_{\xi,\theta}(0)}{\partial_\rho\psi(\xi;\rho_{\xi,\theta}(0),\theta)} d\theta\\
=\frac{2\pi}{\sqrt{3}}\frac{(1+g'(\xi/3)^2)^{3/2}}{g''(\xi/3)}=\frac{2\pi}{\sqrt{3}}\frac{1}{\kappa(\xi/3)}.
\end{multline}
Since $a\geq(\frac{\lambda}{2})^3$, the curvature $\kappa$ attains a local minimum at 0, and so $F(\cdot,0)$ attains a local maximum along the lower boundary of its support. This actually defines a  concave function of $\xi$ in a neighborhood of the origin since
\begin{equation}\label{firstderivativexi}
\frac{\partial}{\partial\xi}\Big|_{\xi=0}F(\xi,0)=0
\end{equation}
and
\begin{equation}\label{secondderivativexi}
\frac{\partial^2}{\partial\xi^2}\Big|_{\xi=0}F(\xi,0)=\frac{2\pi}{\sqrt{3}}\frac{d^2}{d\xi^2}\Big|_{\xi=0}\frac{1}{\kappa(\xi/3)}=\frac{2\pi}{\sqrt{3}}\Big(\frac{\lambda}{3}-\frac{8a}{3\lambda^2}\Big)<0.
\end{equation}

\noindent From \eqref{secondderivative}, \eqref{mixedpartials} and \eqref{secondderivativexi} we have that
\begin{displaymath}
D^2F(0,0)=\frac{2\pi}{\sqrt{3}}\frac{1}{\lambda}
\left(
\begin{array}{cc}
\frac{\lambda^2}{3}-\frac{8 a}{3 \lambda} & 0 \\
0  &  2\lambda-\frac{8 a}{\lambda^2}
\end{array}\right).
\end{displaymath}

\noindent Using the second partial derivative test we conclude that, working as we are under the assumption $a>2(\frac{\lambda}{2})^3$,  the function $F$ attains a strict local maximum at the origin. As discussed before, it follows that the triple convolution ${\sigma}\ast{\sigma}\ast{\sigma}$ also has a strict local maximum at the origin. Choosing $r>0$ to be sufficiently small, one can further ensure  this maximum to be  a {\em global} one. 
\vspace{.5cm}

This is the crucial ingredient in showing that extremizers for inequality \eqref{specialcaprestriction} do not exist  for {sufficiently small }caps $[-r,r]$. The proof goes along the lines of what was done in \cite{F, Q2}, and we recall it here. 
Let $Tf:=\widehat{f\sigma}$.
As in the proof of \cite[Proposition 24]{OS},

\begin{align}
\|Tf\|_6^6&=(2\pi)^2\iint_{\R^2} |f\sigma\ast f\sigma\ast f\sigma(\xi,\tau)|^2d\xi d\tau\label{upperboundforopnormT}\\
&\leq(2\pi)^2\iint |f|^2\sigma\ast |f|^2\sigma\ast |f|^2\sigma(\xi,\tau)\cdot\sigma\ast \sigma\ast \sigma(\xi,\tau)d\xi d\tau\notag\\
&\leq(2\pi)^2\sup_{(\xi,\tau)\in\textrm{ supp}(\sigma^{(\ast 3)})}\sigma^{(\ast 3)} (\xi,\tau)\cdot\iint |f|^2\sigma\ast |f|^2\sigma\ast |f|^2\sigma(\xi,\tau)d\xi d\tau\notag\\
&=(2\pi)^2\|\sigma^{(\ast 3)}\|_{L^\infty(\R^2)} \|f\|_{L^2(\sigma)}^6,\notag
\end{align}
where we used H\"{o}lder's inequality twice. If $f\neq 0$, then all the inequalities become equalities only if $\sigma^{(\ast 3)}(\xi,\tau)=\|\sigma^{(\ast 3)}\|_{L^\infty(\R^2)}$ for a.e. $(\xi,\tau)$ in the support of $\sigma^{(\ast 3)}$; see \cite[Lemma 3.1]{Q2}.

As before, assume that $a>2(\frac{\lambda}{2})^3$ and take $r_0>0$ sufficiently small to ensure that $\sigma^{(\ast 3)}$ attains a strict {\em global} maximum at the origin $(0,0)$ if $r<r_0$. Then
\begin{equation}\label{linftynormoftripleconvolution}
\|\sigma^{(\ast 3)}\|_{L^\infty(\R^2)}=\sigma^{(\ast 3)}(0,0)=\frac{2\pi}{\sqrt{3}}\frac{1}{\lambda}.
\end{equation}

\noindent This quantity is of geometric significance. Indeed, 
\begin{equation}\label{Foschi} 
\frac{2\pi}{\sqrt{3}}\frac{1}{\lambda}=\frac{{\mathbf C_F}[\lambda]^6}{(2\pi)^2},
\end{equation}
where ${\mathbf C_F}[\lambda]$ denotes the optimal constant computed by Foschi \cite{F} for the adjoint Fourier restriction inequality for the parabola
$$\mathbb{P}_\lambda=\Big\{(y,z)\in\R^2: z=\frac{\lambda y^2}{2}\Big\}$$
which osculates the curve $\Gamma$ at the point  $p_0=(0,0)$, equipped with projection measure.

As in \cite[Lemma 4.4]{Q2}, the lower bound $\|T \|\geq {\mathbf {C}}_F[\lambda]$ can be easily verified with the help of an explicit extremizing sequence. For that purpose, consider the scaled Gaussian $G(y)=e^{-\lambda y^2/2}$ and its evolution via the Schr\"{o}dinger flow
$$G_1(x,t)=\int_\R G(y) e^{-it\frac{\lambda y^2}{2}} e^{ixy}dy.$$ 
For $\delta>0$, consider the family of trial functions $f_\delta(y)=\delta^{-1/2} G(\delta^{-1}y)\in L^2(\sigma)$. It is shown in the course of the variational calculation of \cite[\S 8.3]{OS} that
$$\lim_{\delta\rightarrow 0^+}\|T f_\delta\|_{L^6(\R^2)} \|f_\delta\|_{L^2(\sigma)}^{-1}=\|G_1\|_6\|G\|_2^{-1}={\mathbf C_F}[\lambda].$$

Aiming at a contradiction, let $0\neq f\in L^2(\sigma)$ be an extremizer for inequality \eqref{specialcaprestriction}. In particular, $\|T f\|_6=\|T \|\cdot \|f\|_{L^2(\sigma)}$. Using the lower bound just described, and the upper bound given by the chain of inequalities \eqref{upperboundforopnormT}, we get that
\begin{equation}\label{chainofinequalities}
 {\mathbf C_F}[\lambda]\|f\|_{L^2(\sigma)}\leq\|T \|\cdot \|f\|_{L^2(\sigma)}=\|T f\|_6\leq(2\pi)^{1/3}\|\sigma^{(\ast 3)}\|_{L^\infty(\R^2)}^{1/6} \|f\|_{L^2(\sigma)}.
 \end{equation}
The $L^\infty$ norm of the triple convolution $\sigma^{(\ast 3)}$ has been computed in \eqref{linftynormoftripleconvolution}. Using it together with \eqref{Foschi}, we see that all inequalities in \eqref{chainofinequalities} are equalities. As mentioned before, this forces $\sigma^{(\ast 3)}(\xi,\tau)=\|\sigma^{(\ast 3)}\|_\infty$ for a.e. $(\xi,\tau)$ in the support of $\sigma^{(\ast 3)}$. But this cannot happen since $\sigma^{(\ast 3)}$ has a {\em strict} local maximum at the origin. The contradiction shows that extremizers for inequality \eqref{specialcaprestriction} cannot exist.

To finish the proof of Theorem \ref{extremizersdonotexistmodified}, let $\{f_n\}$ be an extremizing sequence for inequality \eqref{specialcaprestriction}. Then $\{|f_n|\}$ is an extremizing sequence of nonnegative functions for the same inequality.
As a consequence of the dichotomy established in  \cite[Proposition 21]{OS} and the conclusion of \cite[Lemma 23]{OS}, the sequence $\{|f_n|\}$ must concentrate at some point $(y,g(y))\in\Gamma$, and so $\{f_n\}$ concentrates at the same point. From \cite[Corollary 25]{OS} it follows that $y=0$.\qedhere
\end{proof}
\vspace{.5cm}

Compare with what we already knew from \cite{OS}: $(\kappa''\cdot\kappa^{-3}) (0)<\frac{3}{2}$ if and only if $a<\frac{3}{2}(\frac{\lambda}{2})^3$. In this case, ${\mathbf C}>{\mathbf C_F}[\lambda]$, and so extremizing sequences for inequality \eqref{specialcaprestriction} cannot concentrate. Precompactness of extremizing sequences of nonnegative functions is ensured in this case, and extremizers exist.
The natural question is then:
 
\begin{question}
Do extremizers exist when the parameter $a$ lies in the interval $[\frac{3}{2}(\frac{\lambda}{2})^3,{2}(\frac{\lambda}{2})^3]$? 
\end{question}

\end{document}